\documentclass[11pt]{amsart}

\usepackage{amssymb}
\usepackage{mathrsfs}
\usepackage{amsfonts}
\usepackage{latexsym,amsmath,amsthm,amssymb,amsfonts}
\usepackage[usenames]{color}
\usepackage{xspace,colortbl}
\usepackage{graphicx}
\usepackage{tipa}

\newcommand{\be}{\begin{equation}}
\newcommand{\ee}{\end{equation}}
\newcommand{\beq}{\begin{eqnarray}}
\newcommand{\eeq}{\end{eqnarray}}

\newtheorem{prop}{Proposition}[section]

\newtheorem{remark}[prop]{Remark}

\def\begeq{\begin{equation}}
\def\endeq{\end{equation}}

\def\tr{{\rm tr}}

\def\odot{\setbox0=\hbox{$\bigcirc$}\relax \mathbin {\hbox
to0pt{\raise.5pt\hbox to\wd0{\hfil $\wedge$\hfil}\hss}\box0 }}

\numberwithin{equation} {section}

\numberwithin{equation}{section}
\newtheorem{theorem}{\bf Theorem}[section]
\newtheorem{proposition}[theorem]{\bf Proposition}
\newtheorem{definition}[theorem]{\bf Definition}
\newtheorem{lemma}[theorem]{\bf Lemma}

\allowdisplaybreaks

\begin{document}
\title[curvature estimates for spacelike graphic hypeprsurfaces]
 {Curvature estimates for spacelike graphic hypersurfaces \\in Lorentz-Minkowski space $\mathbb{R}^{n+1}_{1}$}

\author{
 Ya Gao$^{1}$,~~Jie Li$^{1}$,~~Jing Mao$^{1,\ast}$,~~Zhiqi Xie$^{2}$}

\address{
$^{1}$Faculty of Mathematics and Statistics, Key Laboratory of
Applied Mathematics of Hubei Province, Hubei University, Wuhan
430062, China }

\address{
$^{2}$School of Mathematics and Statistics, Yulin University, Yulin,
719000, China }

\email{Echo-gaoya@outlook.com, 2786201989@qq.com, jiner120@163.com,
zhiqi219@126.com}

\thanks{$\ast$ Corresponding author}

\date{}
\maketitle
\begin{abstract}
In this paper, we can obtain curvature estimates for spacelike
admissible graphic hypersurfaces in the $(n+1)$-dimensional
Lorentz-Minkowski space $\mathbb{R}^{n+1}_{1}$, and through which
the existence of spacelike admissible graphic hypersurfaces, with
prescribed $2$-th Weingarten curvature and Dirichlet boundary data,
defined over a strictly convex domain in the hyperbolic plane
$\mathscr{H}^{n}(1)\subset\mathbb{R}^{n+1}_{1}$ of center at origin
and radius $1$, can be proven.
\end{abstract}

\maketitle {\it \small{{\bf Keywords}: Spacelike hypersurfaces,
Lorentz-Minkowski space, curvature estimates, Dirichlet boundary
condition.}

}

\section{Introduction} \label{se1}

Throughout this paper, let $\mathbb{R}^{n+1}_{1}$ be the
$(n+1)$-dimensional ($n\geq2$) Lorentz-Minkowski space with the
following Lorentzian metric
\begin{eqnarray*}
\langle\cdot,\cdot\rangle_{L}=dx_{1}^{2}+dx_{2}^{2}+\cdots+dx_{n}^{2}-dx_{n+1}^{2}.
\end{eqnarray*}
In fact, $\mathbb{R}^{n+1}_{1}$ is an $(n+1)$-dimensional Lorentz
manifold with index $1$. Denote by
\begin{eqnarray*}
\mathscr{H}^{n}(1)=\{(x_{1},x_{2},\cdots,x_{n+1})\in\mathbb{R}^{n+1}_{1}|x_{1}^{2}+x_{2}^{2}+\cdots+x_{n}^{2}-x_{n+1}^{2}=-1~\mathrm{and}~x_{n+1}>0\},
\end{eqnarray*}
which is exactly the hyperbolic plane\footnote{~The reason why we
call $\mathscr{H}^{n}(1)$ a hyperbolic plane is that it is a
simply-connected Riemannian $n$-manifold with constant negative
curvature and is geodesically complete.} of center $(0,0,\ldots,0)$
(i.e., the origin of $\mathbb{R}^{n+1}$) and radius $1$ in
$\mathbb{R}^{n+1}_{1}$. Clearly, from the Euclidean viewpoint,
$\mathscr{H}^{2}(1)$ is one component of a hyperboloid of two
sheets.

Assume that
 \begin{eqnarray} \label{G1}
\mathcal{G}:=\{(x,u(x))|x\in M^{n} \subset \mathscr{H}^{n}(1)\}
\end{eqnarray}
is a spacelike graphic hypersurface defined over some bounded piece
$M^{n} \subset \mathscr{H}^{n}(1)$, with the boundary $\partial
M^{n}$, of the hyperbolic plane $\mathscr{H}^{n}(1)$, where
$\sup_{M^{n}}\frac{|Du|}{u}\leq\rho<1$. Let $x$ be a point on
$\mathscr{H}^{n}(1)$ which is described by local coordinates
$\xi^{1},\ldots,\xi^{n}$, that is, $x=x(\xi^{1},\ldots,\xi^{n})$. By
the abuse of notations, let $\partial_i$ be the corresponding
coordinate vector fields on $\mathscr{H}^{n}(1)$ and
$\sigma_{ij}=g_{\mathscr{H}^{n}(1)}(\partial_i,\partial_j)$ be the
induced Riemannian metric on $\mathscr{H}^{n}(1)$. Of course,
$\{\sigma_{ij}\}_{i,j=1,2,\ldots,n}$ is also the metric on
$M^{n}\subset\mathscr{H}^{n}(1)$. Denote by\footnote{~Clearly, for
accuracy, here $D_{i}u$ should be $D_{\partial_{i}}u$. In the
sequel, without confusion and if needed, we prefer to simplify
covariant derivatives like this. In this setting,
$u_{ij}:=D_{j}D_{i}u$, $u_{ijk}:=D_{k}D_{j}D_{i}u$ mean
$u_{ij}=D_{\partial_{j}}D_{\partial_{i}}u$ and
$u_{ijk}=D_{\partial_{k}}D_{\partial_{j}}D_{\partial_{i}}u$,
respectively. } $u_{i}:=D_{i}u$, $u_{ij}:=D_{j}D_{i}u$, and
$u_{ijk}:=D_{k}D_{j}D_{i}u$ the covariant derivatives of $u$ w.r.t.
the metric $g_{\mathscr{H}^{n}(1)}$, where $D$ is the covariant
connection on $\mathscr{H}^{n}(1)$. Let $\nabla$ be the Levi-Civita
connection of $\mathcal{G}$ w.r.t. the metric
$g:=u^{2}g_{\mathscr{H}^{n}(1)}-dr^{2}$ induced from the Lorentzian
metric $\langle\cdot,\cdot\rangle_{L}$ of $\mathbb{R}^{n+1}_{1}$.
Clearly, the tangent vectors of $\mathcal{G}$ are given by
\begin{eqnarray*}
X_{i}=(1,Du)=\partial_{i}+u_{i}\partial_{r}, \qquad i=1,2,\ldots,n.
\end{eqnarray*}
The induced metric $g$ on $\mathcal{G}$ has the form
\begin{equation*}\label{g_{ij}}
g_{ij}=\langle X_{i},X_{j}\rangle_{L}=u^2\sigma_{ij}-u_{i}u_{j},
\end{equation*}
its inverse is given by
\begin{equation*}\label{g^{ij}}
g^{ij}=\frac{1}{u^2}\left(\sigma^{ij}+\frac{u^i  u^j
}{u^2v^{2}}\right),
\end{equation*}
and the future-directed timelike unit normal of $\mathcal{G}$ is
given by
\begin{eqnarray*}\label{nu}
\nu=\frac{1}{v}\left(\partial_r+\frac{1}{u^2}u^j\partial_j\right),
\end{eqnarray*}
where $u^{j}:=\sigma^{ij}u_{i}$ and $v:=\sqrt{1-u^{-2}|D u|^2}$ with
$D u$ the gradient of $u$. Of course, in this paper we use the
Einstein summation convention -- repeated superscripts and
subscripts should be made summation from $1$ to $n$. The second
fundamental form of $\mathcal{G}$ is
\begin{equation}\label{h_{ij}}
h_{ij}=-\langle\overline{\nabla}_{X_{j}}X_{i},\nu\rangle_{L}
=\frac{1}{v}\left(u_{ij}+u \sigma_{ij}-\frac{2}{u}{u_i u_j }\right),
\end{equation}
 with $\overline{\nabla}$ the covariant connection in
 $\mathbb{R}^{n+1}_{1}$.
Denote by $\lambda_{1},\lambda_{2},\ldots,\lambda_{n}$ the principal
curvatures of $\mathcal{G}$, which are actually the eigenvalues of
the matrix $(h_{ij})_{n\times n}$ w.r.t. the metric $g$. The
so-called \emph{$k$-th Weingarten curvature} at
$X=(x,u(x))\in\mathcal{G}$ is defined as
 \begin{eqnarray}  \label{kwc}
\sigma_{k}(\lambda_{1}, \lambda_{2}, \cdots,
\lambda_{n})=\sum\limits_{1\leq i_{1}<i_{2}<\cdots<i_{k}\leq
n}\lambda_{i_{1}}\lambda_{i_{2}}\cdots\lambda_{i_{k}}.
 \end{eqnarray}

\begin{remark}
\rm{ (1) Clearly,
$\sigma_{1}=\lambda_{1}+\lambda_{2}+\cdots+\lambda_{n}$ is actually
the mean curvature $H$ of $\mathcal{G}$ at $X$, while
$\sigma_{n}=\lambda_{1}\lambda_{2}\cdots\lambda_{n}$ denotes the
Gauss-Kronecker curvature of $\mathcal{G}$ at $X$. Since
$\mathcal{G}$ is a spacelike hypersurface in $\mathbb{R}^{n+1}_{1}$,
when $n=2$ the intrinsic Gauss curvature of $\mathcal{G}$ at $X$
should be
$-\sigma_{n}$. \\
(2)  As
 explained and shown by L\'{o}pez \cite{rl}, (in suitable orientation) the mean curvature $H$
 of a surface in $\mathbb{R}^{3}_{1}$
 satisfies\footnote{~Provided the dimension constant is neglected.}
 $H=\epsilon\mathrm{tr}(A)$, where $\epsilon=-1$ if the surface
 is spacelike while $\epsilon=1$ if the surface
 is timelike, and $\mathrm{tr}(A)$ stands for the trace of the second fundamental form
 $A$. However, in his setting, each component $h_{ij}$ of $A$ has
 exactly the opposite sign with the one we have used here (i.e., $h_{ij}=\langle\nabla_{X_{j}}X_{i},\nu\rangle_{L}$ in
 \cite{rl}). But, if we use L\'{o}pez's setting here, for the spacelike
 graphic
 hypersurface $\mathcal{G}$, the mean curvature $H$ is the same
 with our treatment here since $\epsilon=-1$ and $H=-\mathrm{tr}(A)$. Hence, there is no essential difference
 between our setting here and L\'{o}pez's. One might find that
 for curves and surfaces in $\mathbb{R}^{3}_{1}$, L\'{o}pez's setting is more convenient
 than the one we have used here. Both settings have been used by us
 in previous works -- see, e.g., \cite{GaoY,gm3} for the setting
 here and \cite{gm2,glm} for L\'{o}pez's. \\
 (3) In \cite{GaoY2}, Gao and Mao \emph{firstly} considered the
 evolution of spacelike graphic hypersurface, defined over a convex
 piece of $\mathscr{H}^{n}(1)$ and contained in a time cone  in
 $\mathbb{R}^{n+1}_{1}$ ($n\geq2$), along the inverse mean curvature flow (IMCF for short) with
 zero Neumann boundary condition (NBC for short), and showed
 that this flow exists for all the time, the spacelike graphic
 property of the evolving hypersurfaces is preserved along flow, and
 after suitable rescaling, the rescaled hypersurfaces converge to a
 piece of
 the spacelike graph of a constant function defined over
 $\mathscr{H}^{n}(1)$ as time tends to infinity.
 Recently, the anisotropic
 versions of this conclusion (both in $\mathbb{R}^{n+1}_{1}$ and more general Lorentz manifold $M^{n}\times\mathbb{R}$) have been solved (see
 \cite{gm2,gm21}). Besides, the lower dimensional case  has also
 been discussed (see \cite{glm}). If the IMCF in \cite{GaoY2} was
 replaced by the inverse Gauss curvature flow (IGCF for short), we can obtain the
 long-time existence and the asymptotical behavior of the new flow (see
 \cite{gm3}). There is one more thing we would like to mention here -- as
 revealed in (3) of \cite[Remark 1.1]{GaoY2}, although a
 new setting for the mean curvature\footnote{~Also different from the one here.} (different from L\'{o}pez's mentioned in (2)
 above) has been used therein, but for the flow problem considered
 in \cite{GaoY2}
 there would not have essential difference between two
 settings if opposite orientations were used for the timelike unit
 normal vector in the IMCF equation. This kind of phenomenons
 happens in the research of Differential Geometry. For instance, one might find that
 there at least exist two definitions for the $(1,3)$-type curvature tensor on
 Riemannian manifolds, which have opposite sign, but \emph{essentially} same
 fundamental equations (such as the Gauss equation, the Codazzi equation, the Ricci identity,
 etc) can be derived provided necessary settings have been made. \\
 (4) One can easily find that boring trouble on sign would happen if one
 uses L\'{o}pez's setting in \cite{rl} (for the second fundamental form, the mean
 curvature, etc) to deal with the prescribed curvature problems in
 $\mathbb{R}^{n+1}_{1}$. Based on this reason, we prefer to go back
 to our treatment in \cite{GaoY} whose definitions for $h_{ij}$ and
 $H$ are the same with ones here. Through this philosophy, we use
 the setting $\sigma_{n}=\lambda_{1}\lambda_{2}\cdots\lambda_{n}$
 for the Gauss-Kronecker curvature in our study of IGCF with zero
 NBC in $\mathbb{R}^{n+1}_{1}$. Of course, in this situation, the
 orientation for the timelike unit normal vector in the flow
 equation should be past-directed.
}
\end{remark}

We also need the following conception:
\begin{definition}
For $1\leq k\leq n$, let $\Gamma_{k}$ be a cone in $ \mathbb{R}^{n}$
determined by
\begin{eqnarray*}
\Gamma_{k}=\{\lambda\in\mathbb{R}^{n}|\sigma_{l}(\lambda)>0,
~~l=1,2,\ldots,k\}.
\end{eqnarray*}
A smooth spacelike graphic hypersurface
$\mathcal{G}\subset\mathbb{R}^{n+1}_{1}$ is called $k$-admissible if
at every point $X\in\mathcal{G}$,
 $(\lambda_{1},\lambda_{2},\ldots,\lambda_{n})\in\Gamma_{k}$.
\end{definition}

In this paper, we investigate the curvature estimates and then the
existence of solutions for a class of nonlinear partial differential
equations (PDEs for short) given as follows
\begin{equation}\label{main equations}
\left\{
\begin{aligned}
&\sigma_{k}=\psi(x,u,\vartheta), \qquad &&x\in M^{n}\subset
\mathscr{H}^{n}(1)\subset\mathbb{R}^{n+1}_{1},~~k=1,2,\ldots,n,
\\
&u=\varphi, \qquad&&x\in \partial M^{n},
\end{aligned}
\right.
\end{equation}
where $\psi$, depending on $X$, $\vartheta:=-\langle X,
\nu\rangle_{L}$, and $\varphi$ are functions defined on $M^{n}$. The
regularity requirements on functions $\psi$ and $\varphi$ would be
mentioned in curvature estimates below. Obviously, by
(\ref{h_{ij}}), we know that $\sigma_{k}$ in (\ref{main equations})
should be determined by the graphic function $u$ and its
derivatives. Based on this fact, if necessary, sometimes we also
write $\sigma_{k}$ as $\sigma_{k}[u]$ to emphasize this connection.
This simplification will be used similarly in the sequel.

\begin{remark}
\rm{ (1) Clearly, (\ref{main equations}) is a prescribed curvature
problem (PCP for short) with Dirichlet boundary condition (DBC for
short). It is reasonable and feasible to consider the PCP
\begin{eqnarray} \label{ME2}
\sigma_{k}=\psi(x,u,\vartheta)
\end{eqnarray}
over $\mathscr{H}^{n}(1)$ or a piece of it. In fact, (i) if $k=1$
 and $\psi=a$ for some positive constant $a>0$ in (\ref{ME2}), then
 $\mathcal{G}$ should be $\mathscr{H}^{n}(\frac{n}{a})$ or a piece of
 it; (ii) if $k=n$
 and $\psi=a>0$ in (\ref{ME2}), then $\mathcal{G}$ should be $\mathscr{H}^{n}(\frac{1}{\sqrt[n]{a}})$
 or a piece of
 it. Obviously, in these two cases, the graphic function $u(x)$ should
 be constant. Naturally, one might try to know more except these
 relatively simple examples. \\
 (2) Assume that $\Omega\subset\mathbb{R}^{n}$ is smooth bounded and
 strictly
 convex, and that $\psi$ is a smooth positive function. For spacelike graphic
 hypersurfaces $\widetilde{G}:=\{(x,u(x))\in\mathbb{R}^{n+1}_{1}|x\in\Omega\}$ defined over $\Omega\subset\mathbb{R}^{n}$,
 Huang \cite{yh}
 considered the following PCP
\begin{equation} \label{me11}
\left\{
\begin{aligned}
&\sigma_{k}=\psi(x,u,w), \qquad &&x\in \Omega,
\\
&u=\varphi, \qquad&&x\in \partial\Omega,
\end{aligned}
\right.
\end{equation}
where $w=1/\sqrt{1-|Du|^{2}}$, and showed the existence of solutions
to (\ref{me11}) provided $\varphi$ is spacelike, affine and
$\psi^{\frac{1}{k}}(x,u,w)$ has extra growth assumption and
convexity in $w$. It is easy to know that the future-directed
timelike unit normal vector $\widetilde{\nu}$ of spacelike graphic
 hypersurfaces $\widetilde{G}$ therein should be
\begin{eqnarray*}
\widetilde{\nu}=\frac{\partial_r+u^{i}\delta_{ij}\partial_j}{\sqrt{1-|Du|^{2}}}=\frac{(Du,1)}{\sqrt{1-|Du|^{2}}},
\end{eqnarray*}
and $w=-\langle\epsilon_{n+1},\widetilde{\nu}\rangle_{L}$ with
$\epsilon_{n+1}=(0,\ldots,0,1)$ the unit basis of the $x_{n+1}$-axis
of $\mathbb{R}^{n+1}_{1}$. This interesting fact leads to an
observation:

\begin{itemize}

\item \emph{Although a spacelike graphic hypersurface defined
over $M^{n}\subset\mathscr{H}^{n}(1)$ is also spacelike graphic over
$\Omega\subset\mathbb{R}^{n}$ and vice versa, since there exists at
least a diffeomorphism between $\Omega$ and $M^{n}$. However, $w$
cannot equal to $\vartheta$ identically by this diffeomorphism.
Therefore, essentially the PCP (\ref{main equations}) should be
different from Huang's (\ref{me11}).}

\end{itemize}
(3) The PCPs (with or without boundary condition) in Euclidean space
or even more general Riemannian manifolds were extensively studied
-- see, e.g., \cite{cns1,cns2,gm,WJJ} and the references therein for
details. Affected by the study of Geometry of Submanifolds, it is
natural to consider PCPs in the pseudo-Riemannian context. In fact,
except Huang's interesting result mentioned above, many other
important results on PCPs in pseudo-Riemannian manifolds have been
obtained. For instance, in the Lorentz-Minkowski space or general
Lorentz manifolds, Bartnik \cite{b1}, Bartnik-Simon \cite{bs},
Gerhardt \cite{cg1,cg2} solved the Dirichlet problem for the
prescribed mean curvature equation, Delano\`{e} \cite{fd}, Guan
\cite{gb} considered the prescribed Gauss-Kronecker curvature
equation with DBC, while Bayard \cite{Bayard}, Gerhardt \cite{cg3},
Urbas \cite{ju2} worked for the prescribed scalar curvature
equation. }
\end{remark}

For the PCP (\ref{main equations}), first, we can get the following
curvature estimate:
\begin{theorem}\label{main1.1}
 Suppose that $u\in C^{4}(M^{n})\cap C^{2}(\overline{M^{n}} )$ is a spacelike, $k$-admissible solution of the PCP $\mathrm{(\ref{main equations})}$, $0<\psi\in C^{\infty}(\overline{M^{n}}) $ and that $\psi^{\frac{1}{k}}(X,\vartheta)$ is convex in $\vartheta$ and satisfies
\begin{equation}\label{f's condition}
\frac{\partial \psi^{\frac{1}{k}}(X,\vartheta)}{\partial
\vartheta}\cdot \vartheta\geq\psi^{\frac{1}{k}}(X,\vartheta) \qquad
for~ fixed~ X\in\mathcal{G}.
\end{equation}
Then the second fundamental form $A$ of $\mathcal{G}$ satisfies
\begin{equation}\label{A's boundary}
\sup\limits_{M^{n}}||A||\leq C\left(1+\sup\limits_{\partial
M^{n}}||A||\right),
\end{equation}
where $C$ depends only on $n$, $||\varphi||_{C^{1}(\overline{M^{n}}
)}$, $||\psi||_{C^{2}\left(\overline{M^{n}}\times\left[\inf\limits_{\partial M^{n}}u,\sup\limits_{\partial M^{n}}u\right]\times\mathbb{R}\right)}$.\\
\end{theorem}

\begin{remark}
\rm{ It is not hard to find some $\psi$ satisfying assumptions in
Theorem \ref{main1.1}. For instance, (i)
$\psi(x,u,\vartheta)=\vartheta^{p}h(x,u)$ for $p\geq k$; (ii)
$\psi(x,u,\vartheta)=e^{p\vartheta}h(x,u)$ for $p\geq k$. }
\end{remark}

An interior curvature estimate can be obtained in the case that
$\varphi$ is affine and satisfies the strict version of (\ref{f's
condition}).

\begin{theorem}\label{main 1.2}
Suppose that $u\in C^{4}(M^{n})\cap C^{2}(\overline{M^{n}})$ is a
spacelike, $k$-admissible solution of the PCP (\ref{main
equations}), $0<\psi\in C^{\infty}(\overline{M^{n}})$ and that
$\psi^{\frac{1}{k}}(X,\vartheta)$ is convex in $\vartheta$ and
satisfies
\begin{equation}\label{f's strictly condition}
\frac{\partial \psi^{\frac{1}{k}}(X,\vartheta)}{\partial
\vartheta}\cdot \vartheta>\psi^{\frac{1}{k}}(X,\vartheta) \qquad
for~ fixed~ X\in\mathcal{G}.
\end{equation}
Furthermore, suppose that $M^{n} \subset \mathscr{H}^{n}(1)$ is
$C^{2}$ and uniformly convex, and that $\varphi$ is spacelike and
affine. If $u\in C^{4}(M^{n})$ is a spacelike, $k$-admissible
solution of the PCP (\ref{ME2}), then
\begin{equation*}
\sup\limits_{\widetilde{M^{n}}}|A|\leq C(\widetilde{M^{n}})
\end{equation*}
for any $\widetilde{M^{n}}\subset\subset M^{n}$, where
$C(\widetilde{M^{n}})$ depends only on $n$, $\zeta$, $M^{n}$,
$\mathrm{dist}(\widetilde{M^{n}}, \partial M^{n})$,
$||\varphi||_{C^{1}(\overline{M^{n}})}$ and
$||\psi||_{C^{2}\left(\overline{M^{n}}\times\left[\inf\limits_{\partial
M^{n}}u,\sup\limits_{\partial
M^{n}}u\right]\times\mathbb{R}\right)}$.
\end{theorem}

\begin{remark}
\rm{ (1) The positive constant $\zeta$ here will be determined
clearly in
the proof of Theorem \ref{main 1.2} in Subsection \ref{SU3-2}. \\
(2) Here, $\mathrm{dist}(\widetilde{M^{n}}, \partial M^{n})$
characterizes the Riemannian distance between $\widetilde{M^{n}}$
and $\partial M^{n}$, and of course, depends on the induced metric
$\{\sigma_{ij}\}_{i,j=1,2,\ldots,n}$ on $\mathscr{H}^{n}(1)$. }
\end{remark}

Combining the above curvature estimates and the $C^{2}$ boundary
estimates shown in \cite[Section 6]{ggm}, together with the method
of continuity, we can get the existence and uniqueness of solutions
to the PCP (\ref{main equations}) with $k=2$ as follows:

\begin{theorem}\label{main 1.3}
Suppose that $M^{n}$ is a smooth bounded domain of
$\mathscr{H}^{n}(1)$ and is strictly convex, while $\psi$ is a
smooth positive function and $\psi^{\frac{1}{2}}$ is convex in
$\vartheta$ satisfying
\begin{equation*}\label{f's condition 2}
\frac{\partial \psi^{\frac{1}{2}}(x, u, \vartheta)}{\partial
\vartheta}\cdot \vartheta\geq\psi^{\frac{1}{2}}(x, u, \vartheta)
\qquad for~~fixed~~(x, u)\in M^{n}\times\mathbb{R}.
\end{equation*}
Then for any spacelike, affine function $\varphi$, there exists a
uniquely smooth spacelike, $2$-admissible graphic hypersurface
$\mathcal{G}$ (defined over $M^{n}$) with the prescribed curvature
$\psi$ and Dirichlet boundary data $\varphi$.
\end{theorem}

\begin{remark} \label{remark1.5}
\rm{ (1) In the PCP (\ref{main equations}), if
$\sigma_{k}=\sigma_{k}(\lambda(A))$ was replaced
by\footnote{~Clearly, in (1) of Remark \ref{remark1.5} here,
$\sigma_{k}(\lambda(\cdot))$ denotes the $k$-th elementary symmetric
function of eigenvalues of a given tensor -- the second fundamental
form $A$.}
$$\frac{\sigma_{k}(\lambda(A))}{\sigma_{l}(\lambda(A))}$$
 with
$2\leq k\leq n$, $0\leq l\leq k-2$,  then the a priori estimates for
solutions to the corresponding
 Dirichlet problem of a class of Hessian quotient equations can be
 obtained under suitable assumptions, which leads to the existence and
uniqueness of solutions for some $k$ -- see \cite{ggm} for
 details.\\
  (2) Clearly, if $l=0$, then the $(k,l)$-Hessian quotient
  $\frac{\sigma_{k}(\lambda(A))}{\sigma_{l}(\lambda(A))}$ becomes
  $\sigma_{k}(\lambda(A))$, which implies that the PCP considered in
  \cite{ggm} covers (\ref{main equations}) as a special case. This
  leads to the fact that the a priori estimates obtained therein, which of
  course is much complicated than the one shown in this paper, can
  be used directly in the usage of Schauder theory in the proof of
  existence of solutions to the PCP (\ref{main equations}) shown in Section
  \ref{se5}. For the purpose of simplification, the $C^{2}$ boundary estimates of the PCP (\ref{main
  equations}) will not be given here, and
  readers can check a more general and more complicated version
  given in \cite[Section 6]{ggm}.
  \\
 (3) We have already shown that it is reasonable and feasible to consider
 PCPs (with DBC) on bounded domains in
 $\mathscr{H}^{n}(1)\subset\mathbb{R}^{n+1}_{1}$ through Theorem \ref{main 1.3} here and
 \cite{ggm}. Based on this fact, one can try to extend the existing
 results on the PCPs to this setting. We prefer to leave this
 attempt to readers who are interested in this topic and we believe
 that
 our work here and \cite{ggm} would give some guidance.

}
\end{remark}

The paper is organized as follows. Some useful formulae for
spacelike graphic hypersurfaces defined over
$M^{n}\subset\mathscr{H}^{n}(1)$ will be introduced in Section
\ref{S2}. Parts of these formulae were shown by us firstly in
\cite{GaoY2} and were also mentioned in some works later (see, e.g.,
\cite{gm2}-\cite{ggm}). In Section \ref{c1es}, we will give the
$C^{1}$ estimate for the PCP (\ref{main equations}). Curvature
estimates in Theorems \ref{main1.1} and \ref{main 1.2} will be
proven in Section \ref{S3}. The proof of Theorem \ref{main 1.3} will
be shown in the last section.

\section{Some Elementary Formulas} \label{S2}
As shown in \cite[Section 2]{GaoY}, we have the following fact:

\textbf{FACT}. Given an $(n+1)$-dimensional Lorentz manifold
$(\overline{N}^{n+1},\overline{g})$, with the metric $\overline{g}$,
and its spacelike hypersurface $N^{n}$. For any $p\in N^{n}$, one
can choose a local  Lorentzian orthonormal frame field
$\{e_{0},e_{1},e_{2},\ldots,e_{n}\}$ around $p$ such that,
restricted to $N^{n}$, $e_{1},e_{2},\ldots,e_{n}$ form orthonormal
frames tangent to $N^{n}$. Taking the dual coframe fields
$\{z_{0},z_{1},z_{2},\ldots,z_{n}\}$ such that the Lorentzian metric
$\overline{g}$ can be written as
$\overline{g}=-z_{0}^{2}+\sum_{i=1}^{n}z_{i}^{2}$. Making the
convention on the range of indices
\begin{eqnarray*}
0\leq I,J,K,\ldots\leq n; \qquad\qquad 1\leq i,j,k\ldots\leq n,
\end{eqnarray*}
and doing differentials to forms $z_{I}$, one can easily get the
following structure equations
\begin{eqnarray}
&&(\mathrm{Gauss~ equation})\qquad \qquad R_{ijkl}=\overline{R}_{ijkl}-(h_{ik}h_{jl}-h_{il}h_{jk}), \label{Gauss}\\
&&(\mathrm{Codazzi~ equation})\qquad \qquad h_{ij,k}-h_{ik,j}=\overline{R}_{0ijk},  \label{Codazzi}\\
&&(\mathrm{Ricci~ identity})\qquad \qquad
h_{ij,kl}-h_{ij,lk}=\sum\limits_{m=1}^{n}h_{mj}R_{mikl}+\sum\limits_{m=1}^{n}h_{im}R_{mjkl},
\label{Ricci}
\end{eqnarray}
where $R$ and $\overline{R}$ are the curvature tensors of $N^{n}$
and $\overline{N}^{n+1}$ respectively. Clearly, in our setting here,
all formulae mentioned above can be used directly with
$\overline{N}^{n+1}=\mathbb{R}^{n+1}_{1}$ and
$\overline{g}=\langle\cdot,\cdot\rangle_{L}$.

For the spacelike graphic hypersurface
$\mathcal{G}\subset\mathbb{R}^{n+1}_{1}$ given by (\ref{G1}) and
$X=(x,u(x))\in\mathcal{G}$, set $X_{,ij}:=\partial_i
\partial_j X-\Gamma_{ij}^{k}X_k$ with $\Gamma_{ij}^{k}$ the
Christoffel symbols of the metric on $\mathcal{G}$. Then it is easy
to know
\begin{eqnarray*}
h_{ij}=-\left\langle X_{,ij}, \nu\right\rangle_{L},
 \end{eqnarray*}
 and have the following identities
\begin{eqnarray}
&&(\mathrm{Gauss~formula}) \qquad \qquad \qquad X_{,ij}=h_{ij}\nu,
\label{Gauss
for} \\
&& (\mathrm{Weingarten~formula}) \qquad \qquad \nu_{,i}=h_{ij}X^{j}.
\label{Wein for}
\end{eqnarray}
Using (\ref{Gauss}), (\ref{Codazzi}) and (\ref{Ricci}) with the fact
$\overline{R}=0$ in our setting, we have
\begin{equation}\label{Gauss-1}
R_{ijkl}=h_{il}h_{jk}-h_{ik}h_{jl},
\end{equation}
\begin{equation}\label{Codazzi-1}
\nabla_{k}h_{ij}=\nabla_{j}h_{ik}, \qquad (i.e.,~h_{ij,k}=h_{ik,j})
\end{equation}
and
\begin{eqnarray}\label{Laplace}
\Delta h_{ij}=(\sigma_{1})_{,ij}-\sigma_{1}
h_{ik}h^{k}_{j}+h_{ij}|A|^{2},
\end{eqnarray}
 where as usual $\nabla$, $\Delta$ denote the gradient and the
 Laplace operators on $\mathcal{G}$, respectively.
Here the comma ``," in subscript of a given tensor means doing
covariant derivatives. Besides, we make an agreement that, for
simplicity, in the sequel the comma ``," in subscripts will be
omitted unless necessary.

\begin{remark}
\rm{ Similar to the Riemannian case, the derivation of the formula
(\ref{Laplace}) depends on equations (\ref{Gauss-1}) and
(\ref{Codazzi-1}).}
\end{remark}

We also need the following fact:

\begin{lemma}\label{sg's formula}
Let $\lambda=(\lambda_{1}, \lambda_{2}, \cdots,
\lambda_{n})\in\mathbb{R}^{n}$ and $k=0, 1, 2, \cdots, n$. Denote by
$\sigma_{k}(\lambda)$ defined as (\ref{kwc}) the $k$-th elementary
symmetric function of $\lambda_{1}, \lambda_{2}, \ldots,
\lambda_{n}$. Also set $\sigma_{0}=1$. Denote by
$\sigma_{k}(\lambda|i)$ the symmetric function with $\lambda_{i}=0$.
Then for any $1\leq i\leq n$, one has
\begin{eqnarray*}\label{sg 1}
\sigma_{k+1}(\lambda)=\sigma_{k+1}(\lambda|i)+\lambda_{i}\sigma_{k}(\lambda|i),
\end{eqnarray*}
\begin{eqnarray*}\label{sg 2}
\sum_{i=1}^{n}\lambda_{i}\sigma_{k}(\lambda|i)=(k+1)\sigma_{k+1},
\end{eqnarray*}
\begin{eqnarray*}\label{sg 3}
\sum_{i=1}^{n}\sigma_{k}(\lambda|i)=(n-k)\sigma_{k}(\lambda),
\end{eqnarray*}
\begin{eqnarray*}\label{sg 4}
\frac{\partial\sigma_{k+1}(\lambda)}{\partial\lambda_{i}}=\sigma_{k}(\lambda|i),
\end{eqnarray*}
and
\begin{eqnarray*}\label{sg 5}
\sum_{i=1}^{n}\lambda_{i}^{2}\sigma_{k}(\lambda|i)=\sigma_{1}(\lambda)\sigma_{k+1}(\lambda)-(k+2)\sigma_{k+2}(\lambda).
\end{eqnarray*}
\end{lemma}

\begin{proof}
The above properties of $\sigma_{k}$ can be obtained by direct
calculations, which we prefer to omit here.
\end{proof}

For any equation
\begin{equation}\label{F's equation}
F(A)=f(\lambda_{1},\lambda_{2},\cdots,\lambda_{n}),
\end{equation}
where $A$ is the second fundamental form of the spacelike graphic
hypersurface $\mathcal{G}\subset\mathbb{R}^{n+1}_{1}$ with
$\lambda_{1},\lambda_{2},\cdots,\lambda_{n}$ its principal
curvatures. We can prove the following two conclusions:

\begin{lemma}\label{F_{ij}'s equation 1}
For the function $F$ defined by (\ref{F's equation}) and the
quantity $\vartheta$ given in the PCP (\ref{main equations}), one
has
\begin{eqnarray*}\label{$F_{ij}$ 1}
&F^{ij}\nabla_{i}\nabla_{j}\nu=\nu
F^{ij}h_{j}^{m}h_{im}+F^{ij}\nabla_{i}h_{j}^{m}X_{m},
\end{eqnarray*}
\begin{eqnarray*}\label{w's}
&\Delta \vartheta=\sigma_{1}+\nabla^{i}\sigma_{1}\langle
X,X_{i}\rangle_{L}+|A|^{2}\vartheta.
\end{eqnarray*}
\end{lemma}

\begin{proof}
By the Weingarten formula (\ref{Wein for}), it follows that
$$\nabla_{i}\nabla_{j}\nu=\nabla_{i}\left(h_{j}^{m}X_{m}\right)=\nabla_{i}h_{j}^{m}X_{m}+h_{j}^{m}h_{im}\nu.$$
The second assertion in Lemma \ref{F_{ij}'s equation 1} can be
obtained as follows
\begin{equation*}
\begin{aligned}
\Delta\vartheta&=g^{mn}\nabla_{m}\nabla_{n}\langle X, \nu\rangle_{L}\\
&=
g^{mn}\nabla_{m}\left(h_{n}^{i}\langle X, X_{i}\rangle_{L}\right)\\
&= \nabla^{i}\sigma_{1}\langle X,
X_{i}\rangle_{L}+\sigma_{1}+|A|^{2}\vartheta.
\end{aligned}
\end{equation*}
by using the Gauss formula (\ref{Gauss for}) and also (\ref{Wein
for}).
\end{proof}

\begin{lemma} \label{F_{ij}'s equation 2}
For the function $F$ defined by (\ref{F's equation}), we have
\begin{equation*}\label{$F_{ij}$ 2}
F^{ij}\nabla_{i}\nabla_{j}\sigma_{1}=-F^{ij,pq}\nabla^{k}h_{ij}\nabla_{k}h_{pq}+F^{ij}h_{j}^{m}h_{im}\sigma_{1}-F^{ij}h_{ij}|A|^{2}+\Delta{f}
 \end{equation*}
 and
\begin{equation*}\label{$F_{ij}$ 3}
F^{ij}\nabla_{i}\nabla_{j}h_{mn}=-F^{ij,pq}\nabla_{n}h_{ij}\nabla_{m}h_{pq}+F^{ij}h_{j}^{l}h_{il}h_{mn}-F^{ij}h_{m}^{l}h_{ln}h_{ij}+\nabla_{m}\nabla_{n}{f}.
 \end{equation*}
 \end{lemma}

 \begin{proof}
 Using (\ref{Laplace}), it
 follows that
 \begin{equation*}
 F^{ij}\nabla_{i}\nabla_{j}\sigma_{1}=F^{ij}h_{j}^{m}h_{im}\sigma_{1}-F^{ij}h_{ij}|A|^{2}+F^{ij}\Delta
 h_{ij}.
 \end{equation*}
On the other hand, by direct calculation, one has
 \begin{equation*}
 \begin{aligned}
 \Delta F&=\Delta f=g^{kl}\nabla_{k}\nabla_{l}F\\
 &=
 g^{kl}\nabla_{k}\left(F^{ij}\nabla_{l}h_{ij}\right)\\
 &=
 F^{ij,pq}\nabla^{k}h_{ij}\nabla_{k}h_{pq}+F^{ij}\Delta h_{ij}.
 \end{aligned}
 \end{equation*}
 The first assertion can be obtained by combining the above two identities.
 The second assertion of Lemma \ref{F_{ij}'s equation 2} can be
 proven similarly.
 \end{proof}

 \begin{remark}
\rm{Clearly, in the proofs of Lemmas \ref{F_{ij}'s equation 1} and
\ref{F_{ij}'s equation 2}, we know that $F^{ij}:=\partial F/\partial
h_{ij}$, $F^{ij,pq}:=\partial^{2}F/\partial h_{ij}\partial h_{pq}$.}
 \end{remark}

\section{$C^{1}$ estimate} \label{c1es}

\subsection{Boundary estimate}
Let $s^{+}$ be the solution of the following Dirichlet
problem\footnote{~Using similar arguments to \cite{bs,fd}, one can
easily get the existence of solutions to the Dirichlet problems
(\ref{dp-31}) and  (\ref{dp-32}) respectively. }
\begin{equation} \label{dp-31}
\left\{
\begin{aligned}
&\sigma_{1}[s]=n\left( \frac{\psi(x,u,\vartheta)}{C_{n}^{k}}
\right)^{\frac{1}{k}}, \qquad &&x\in M^{n},
\\
&s=\varphi, \qquad&&x\in \partial M^{n}.
\end{aligned}
\right.
\end{equation}
From the Mac-Laurin development, we have
$$\sigma_{1}[u] \geq \sigma_{1}[s^{+}].$$
The comparison principle for the mean curvature operator gives
$u\leq s^{+}$ in $M^{n}$, and thus $\frac{\partial u}{\partial \nu}
\geq \frac{\partial s^{+}}{\partial \nu}$. In order to get a lower
barrier, let $s^{-}$ be the solution of the following Dirichlet
problem
\begin{equation} \label{dp-32}
\left\{
\begin{aligned}
&\sigma_{n}[s]=\left( \frac{\psi(x,u,\vartheta)}{C_{n}^{k}}
\right)^{\frac{n}{k}}, \qquad &&x\in M^{n},
\\
&s=\varphi, \qquad&&x\in \partial M^{n}.
\end{aligned}
\right.
\end{equation}
Also from the Mac-Laurin development, we have
$$\sigma_{n}[u] \leq \sigma_{n}[s^{-}].$$
So $u \geq s^{-}$ in $M^{n}$, and thus $\frac{\partial u}{\partial
\nu} \leq \frac{\partial s^{-}}{\partial \nu}$.

\subsection{Maximum principle}

The upper bound on $Du$ amounts to an upper bound on
$\mathcal{W}:=\frac{1}{v}=1/\sqrt{1-|D\pi|^{2}}$, where $\pi:=\ln
u$. Therefore, it would follow from the boundary estimate once one
can prove that $\mathcal{W}e^{S\pi}$ cannot attain an interior maximum for $S$
sufficiently large under control.

\begin{proposition}\label{maximum principle}
Let $u$ be the admissible solution of the PCP \eqref{main
equations}. Then
\begin{equation*}
\sup\limits_{\overline{M^{n}}}\mathcal{W} \leq
\left(\sup\limits_{\partial M^{n}}\mathcal{W}
\right)e^{S_{2}\left(2\sup\limits_{\partial M^{n}}|\varphi| +
\mathrm{diam}(M^{n})\right)},
\end{equation*}
 where as usual $\mathrm{diam}(M^{n})$ stands for the diameter of
 the bounded domain $M^{n}\subset\mathscr{H}^{n}(1)$.
\end{proposition}

\begin{proof}
By contradiction, suppose that
$\sup_{\overline{M^{n}}}\mathcal{W}e^{S\pi}$ is achieved at an
interior point $x_{0}\in M^{n}$. At $x_0$, we choose a nice basis
for the convenience of computations, that is, let $\{e_{1}, e_{2},
\cdots, e_{n}\}$ be an orthonormal basis of $T_{x_0}M^{n}$ (i.e.,
the tangent space at $x_{0}$ diffeomorphic to $\mathbb{R}^{n}$) such
that $D\pi(x_{0}) = |D\pi(x_{0})|e_{1}$, and moreover, the matrix
$\left((D^{2}\pi(x_0))_{ij}\right)_{(n-1)\times(n-1)}$, $2\leq
i,j\leq n$, is orthogonal under the basis $\{e_{2}, \cdots,
e_{n}\}$. Since $|\pi_{1}| \leq |D\pi|$ on $\overline{M^{n}}$ and
$\pi_{1}(x_{0}) = |D\pi(x_{0})|$. The function
$$\ln\left( \frac{1}{\sqrt{1-\pi_{1}^{2}}}\right) + S\pi = -\frac{1}{2}\ln\left( 1-\pi_{1}^{2} \right) + S\pi $$
has a maximum at $x_{0}$ as well. Hence, at $x_{0}$, for any $i\in\{
1,\cdots,n \}$, one has
$$\frac{\pi_{1i}\pi_{1}}{1-\pi_{1}^{2}} + S\pi_{i} = 0.$$
So, the matrix of the curvature operator is diagonal, with diagonal
entries $(\frac{1}{uv}(1+\frac{\pi_{11}}{v^{2}}),
\frac{1}{uv}(1+\pi_{22}), \cdots, \frac{1}{uv}(1+\pi_{nn}) )$.
Moreover, still at $x_{0}$, one has $\pi_{111}\pi_{1} \leq
-\pi_{11}^{2} - \frac{2(\pi_{1}\pi_{11})^{2}}{1-\pi_{1}^{2}} -
S\pi_{11}(1-\pi_{1}^{2})$, and for $i>1$, $\pi_{1ii}\pi_{1} \leq
-(1-\pi_{1}^{2})S\pi_{ii}$. Then we have
\begin{equation*}
\sum_{i=1}^{n}\frac{\partial
\sigma_{k}}{\partial\lambda_{i}}\cdot\lambda_{i,1} =
\sum_{i=1}^{n}\frac{\partial\sigma_{k}}{\partial\lambda_{i}}\cdot
h_{i,1}^{i} = \psi_{1}.
\end{equation*}
Since $h_{i}^{i} = \frac{1}{uv}\left(1+(\sigma^{ik} +
\frac{\pi^{i}\pi^{k}}{v^{2}})\pi_{ik}\right)$, we have
\begin{equation*}
h_{1,1}^{1} = \frac{3\pi_{1}\pi_{11}^{2}}{uv^{5}} +
\frac{\pi_{111}}{uv^{3}} - \frac{\pi_{1}}{uv} =
\frac{\pi_{1}(3S^{2}-1)}{uv} + \frac{\pi_{111}}{uv^{3}},
\end{equation*}
\begin{equation*}
\begin{aligned}
h_{i,1}^{i} &= \frac{\pi_{1}\pi_{11}\pi_{ii}}{uv^{3}} + \frac{\pi_{ii1}}{uv} + \frac{\pi_{1}\pi_{11}}{uv^{3}} - \frac{\pi_{1}}{uv} - \frac{\pi_{1}\pi_{ii}}{uv}\\
&= \frac{\pi_{ii1}}{uv} - \frac{\pi_{1}\pi_{ii}(S+1)}{uv} -
\frac{\pi_{1}(S+1)}{uv} \qquad \mathrm{for}~~ i>1.
\end{aligned}
\end{equation*}
The differentiated equation, multiplied by $\pi_{1}$, becomes:
\begin{equation*}
\begin{aligned}
&\frac{\partial\sigma_{k}}{\partial\lambda_{1}}\left( \frac{\pi_{1}^{2}(3S^{2}-1)}{uv} + \frac{\pi_{111}\pi_{1}}{uv^{3}} \right)\\
&+ \sum_{i\geq2}\frac{\partial\sigma_{k}}{\partial\lambda_{i}}\left(
\frac{\pi_{ii1}\pi_{1}}{uv} - \frac{\pi_{1}^{2}\pi_{ii}(S+1)}{uv} -
\frac{\pi_{1}^{2}(S+1)}{uv} \right) = \pi_{1}\psi_{1}.
\end{aligned}
\end{equation*}
From the maximum conditions, we have
\begin{equation*}
\frac{\pi_{1}^{2}(3S^{2}-1)}{uv} + \frac{\pi_{111}\pi_{1}}{uv^{3}}
\leq \frac{\pi_{1}^{2}(S^{2}-1)}{uv},
\end{equation*}
and, since $\pi_{1ii} = \pi_{ii1} - \pi_{1}$, we have
\begin{equation*}
\begin{aligned}
&\quad
\frac{\pi_{ii1}\pi_{1}}{uv} - \frac{\pi_{1}^{2}\pi_{ii}(S+1)}{uv} - \frac{\pi_{1}^{2}(S+1)}{uv}\\
&\leq -\frac{1}{u}vS\pi_{ii} - \frac{\pi_{1}^{2}S}{uv} -
\frac{\pi_{1}^{2}\pi_{ii}(S+1)}{uv}.
\end{aligned}
\end{equation*}
Then we can infer
\begin{equation*}
\frac{\partial\sigma_{k}}{\partial\lambda_{1}}\cdot\frac{\pi_{1}^{2}(S^{2}-1)}{uv}
- \sum_{i\geq 2}\frac{\partial\sigma_{k}}{\partial\lambda_{i}}\left(
\frac{1}{u}vS\pi_{ii} + \frac{\pi_{1}^{2}S}{uv} +
\frac{\pi_{1}^{2}\pi_{ii}(S+1)}{uv} \right) \geq \pi_{1}\psi_{1},
\end{equation*}
and finally can obtain
\begin{equation*}
-k\sigma_{k}(\pi_{1}^{2} + S) + (n-k+1)\sigma_{k-1}\cdot\frac{2S +
1}{uv} + \frac{\partial\sigma_{k}}{\partial\lambda_{1}}\left(
\frac{vS(1-S)}{u} - \frac{2S+1}{uv} \right) \geq \pi_{1}\psi_{1}.
\end{equation*}
We hope
\begin{equation*}
(n-k+1)\sigma_{k-1}\cdot\frac{2S + 1}{uv} +
\frac{\partial\sigma_{k}}{\partial\lambda_{1}}\left(
\frac{vS(1-S)}{u} - \frac{2S+1}{uv} \right) \leq 0,
\end{equation*}
which is equivalent to
$$\sigma_{k-1}(\lambda)\left( v^{2}S(1-S) + (n+1)(2S+1) \right) \leq 0.$$
Since $\pi_{1}^{2}\leq \rho^{2}<1$, choosing $S=S_{1}$ large enough
such that $\frac{S_{1}(S_{1}-1)}{2S_{1}+1}\geq
\frac{n+1}{1-\rho^{2}}$, so we have
$$k\sigma_{k}S \leq \sup\limits_{\overline{M^{n}}}|D\psi|.$$
Then choosing
$S_{2}>\max\left\{\frac{\sup\limits_{\overline{M^{n}}}|D\psi|}{k\inf\limits_{\overline{M^{n}}}\psi},
S_{1}\right\}$, we reach a contradiction.
\end{proof}

\section{Curvature Estimates} \label{S3}

\subsection{The first curvature estimate}

We write (\ref{main equations}) in the form
\begin{equation}\label{3.1}
F(A)=\sigma_{k}^{\frac{1}{k}}(A)=\psi^{\frac{1}{k}}(X,\vartheta)=f(X,\vartheta)
\qquad \mathrm{for ~ any} ~ X\in\mathcal{G}.
\end{equation}

\begin{proof}[Proof of Theorem \ref{main1.1}]
 Consider the function
\begin{eqnarray*}
W(A)=\sigma_{1}(A),
\end{eqnarray*}
which attains its maximum value at some
$X_{0}=(x_{0},u(x_{0}))\in\mathcal{G}$. If $x_{0}\in\partial M^{n}$,
then our claim (\ref{A's boundary}) follows directly. Now, we try to
prove this claim in the case that $x_{0}\notin\partial M^{n}$.
Choose the frame fields  $e_{1},e_{2},\cdots,e_{n}, \nu$ at $X_{0}$
such that $e_{1},e_{2},\cdots,e_{n}\in T_{X_{0}}\mathcal{G}$ at
$X_{0}$ and $(h_{ij})_{n\times n}$ is diagonal at $X_{0}$ with
eigenvalues $h_{11}\geq h_{22}\geq\cdots\geq h_{nn}$. Here, as
usual, $T_{X_{0}}\mathcal{G}$ denotes the tangent space of the
graphic hypersurface $\mathcal{G}$ at $X_{0}$. For each
$i=1,\ldots,n$, we have
\begin{equation*}\label{3.2}
\nabla_{i}\sigma_{1}=0 \qquad \mathrm{at}~ X_{0}.
\end{equation*}
Therefore, at $X_{0}$, it follows that
\begin{equation}\label{3.3}
\begin{split}
0&\geq F^{ij}\nabla_{i}\nabla_{j}\sigma_{1}\\
&=-F^{ij,pq}\nabla_{l}h_{ij}\nabla_{l}h_{pq}+F^{ij}h_{im}h_{mj}\sigma_{1}-F^{ij}h_{ij}|A|^{2}+\Delta
f.
\end{split}
\end{equation}
Since $f$ is convex in $\vartheta$, together with Lemma
\ref{F_{ij}'s equation 1}, we have
\begin{equation}\label{3.4}
\begin{split}
\Delta f&=\frac{\partial^2 f}{\partial X^{\alpha}\partial X^{\beta}}\nabla_{l}X^{\alpha}\nabla_{l}X^{\beta}+2\frac{\partial^2 f}{\partial X^{\alpha}\partial \vartheta}\nabla_{l}X^{\alpha}\nabla_{l}\vartheta\\
&\quad
+\frac{\partial^2 f}{\partial \vartheta^{2}}|\nabla \vartheta|^{2}+\frac{\partial f}{\partial X^{\alpha}}\Delta X^{\alpha}+\frac{\partial f}{\partial \vartheta}\Delta \vartheta\\
&\geq\frac{\partial f}{\partial \vartheta}\Delta \vartheta+\frac{\partial^2 f}{\partial \vartheta^{2}}|\nabla \vartheta|^{2}-c_{1}\sigma_{1}-c_{2}\\
&\geq\frac{\partial f}{\partial
\vartheta}\vartheta|A|^{2}-c_{1}\sigma_{1}-c_{2},
\end{split}
\end{equation}
where positive constants $c_{1}$, $c_{2}$ depend on
$||\varphi||_{C^{1}(\overline{M^{n}})}$,
$||\psi||_{C^{2}\left(\overline{M^{n}}\times\left[\inf\limits_{\partial
M^{n}}u,\sup\limits_{\partial
M^{n}}u\right]\times\mathbb{R}\right)}$, and $X^{\alpha}:=\langle X,
\partial_{\alpha}\rangle_{L}$, $\alpha=1,2,\ldots,n+1$. Obviously,
$\partial_{1},\partial_{2},\cdots,\partial_{n}$ are the
corresponding coordinate vector fields on $\mathscr{H}^{n}(1)$,
$\partial_{n+1}:=\partial_{r}$. Putting (\ref{3.4}) into (\ref{3.3})
yields
\begin{equation}\label{3.5}
\begin{split}
0&\geq F^{ij}\nabla_{i}\nabla_{j}\sigma_{1}\\
&\geq -F^{ij,pq}\nabla_{l}h_{ij}\nabla_{l}h_{pq}+F^{ij}h_{im}h_{mj}\sigma_{1}\\
&\quad
+(\frac{\partial f}{\partial \vartheta}\vartheta-f)|A|^{2}-c_{1}\sigma_{1}-c_{2}\\
&\geq F^{ij}h_{im}h_{mj}\sigma_{1}-c_{1}\sigma_{1}-c_{2},
\end{split}
\end{equation}
where we have used (\ref{f's condition}) and the concavity of $F$.
On the other hand, by Lemma \ref{sg's formula}, one has
\begin{equation}\label{3.6}
\begin{split}
F^{ij}h_{im}h_{mj}&=\frac{1}{k} \sigma_{k}^{\frac{1}{k} -1}\left[\sigma_{k}\sigma_{1}-(k+1)\sigma_{k+1}\right]\\
&\geq\frac{1}{n}\sigma_{k}^{\frac{1}{k}}\sigma_{1},
\end{split}
\end{equation}
where the last inequality can be derived from the Newton
inequalities for $\sigma_{k+1}>0$,
\begin{eqnarray*}
\frac{\sigma_{k+1}}{C_{n}^{k+1}}\frac{\sigma_{k-1}}{C_{n}^{k-1}}\leq\left(\frac{\sigma_{k}}{C_{n}^{k}}\right)^{2}.
\end{eqnarray*}
Taking (\ref{3.6}) into (\ref{3.5}), it is easy to know that
 $\sigma_{1}$ is bounded. Then the conclusion of
Theorem \ref{main1.1}, i.e. (\ref{A's boundary}), follows naturally.
\end{proof}

\subsection{The second curvature estimate} \label{SU3-2}
Let
$$\mathcal{P}(\lambda):=F(A)=\sigma_{k}^{\frac{1}{k}}(A)=f(X, \vartheta) \qquad \mathrm{for~any}~X\in\mathcal{G}.$$
Set
\begin{equation}\label{set 1}
\sigma_{k}^{\frac{1}{k}}(\lambda_{1}, \cdots,
\lambda_{n})=\mathcal{P}(\lambda_{1}, \cdots, \lambda_{n}),
\end{equation}
\begin{equation}\label{set 3}
\mathrm{tr}F^{ij}=\sum\limits_{i=1}^{n}F^{ii},\qquad
\mathcal{P}_{i}=\frac{\partial \mathcal{P}}{\partial\lambda_{i}}.
\end{equation}
First, we list a useful lemma, which can be found in, e.g.,
\cite{ba,WJJ,ju2}.
\begin{lemma}\label{sym}
For any symmetric matrix $\eta=(\eta_{ij})$, we have
\begin{equation}\label{symmetric}
F^{ij,pq}\eta_{ij}\eta_{pq}=
\sum\limits_{i,j}\frac{\partial^{2}\mathcal{P}}{\partial\lambda_{i}\partial\lambda_{j}}\eta_{ii}\eta_{jj}+\sum\limits_{i\neq
j}\frac{\mathcal{P}_{i}-\mathcal{P}_{j}}{\lambda_{i}-\lambda_{j}}\eta_{ij}^{2}.
\end{equation}
The second term on RHS of (\ref{symmetric}) is nonpositive if
$\mathcal{P}$ is concave, and it is interpreted as the limit if
$\lambda_{i}=\lambda_{j}$.
\end{lemma}

\begin{proof} [Proof of Theorem \ref{main 1.2}]
Let $\eta=\varphi-u$ and, as before, for any point $x_{0}\in M^{n}$,
$X_{0}=(x_{0}, u(x_{0}))$. Denote by $\omega$ the constant function,
whose graph is the hyperbolic plane of center at origin and radius
$R$ (i.e., $\mathscr{H}^{n}(R)$), lying above the graph of $\varphi$
such that $\omega(x_{0})=\varphi(x_{0})$ and
$D\omega(x_{0})=D\varphi(x_{0})$.

Then, for large enough $R$ and small enough $\epsilon>0$, we have
$F\left[(\omega-\epsilon)(A)\right]<F[u(A)]$ in
$M^{n}_{\epsilon}:=\{x\in M^{n}|
\omega(x)-\epsilon<\varphi(x)\}\subset\subset M^{n}$ and
$\omega(x)-\epsilon=\varphi(x)\geq u$ on $\partial
M^{n}_{\epsilon}$. By the comparison principle we then have
$u\leq\omega(x)-\epsilon$ in $M^{n}_{\epsilon}$. Consequently
$(\varphi-u)(x_{0})\geq\epsilon$, so we have $\eta>0$ in $M^{n}$.

We now consider the function
$$G=\eta^{\alpha}e^{\Psi(\vartheta)}h_{ij}\chi_{i}\chi_{j},$$
achieving its maximum value at some $X_{0}\in\mathcal{G}$, where
$\alpha\geq 1$, $\Psi$ is a function determined later and satisfies
$\Psi^{'}:=\frac{\partial\Psi}{\partial\vartheta}\geq 0$. Without
loss of generality,  one may choose the frame fields $e_{1}=\chi$,
$e_{2}$, $\ldots$, $e_{n}$, $\nu$ such that $e_{1}, e_{2}, \cdots,
e_{n}\in T_{X_{0}}\mathcal{G}$, $\nabla_{e_{i}}e_{j}=0$ at $X_{0}$
for all $i, j=1, \ldots, n$, and $(h_{ij})_{n\times n}$ is diagonal
at $X_{0}$ with eigenvalues $h_{11}\geq h_{22}\geq\cdots\geq
h_{nn}$. At $X_{0}$, for each $i=1, \cdots, n$, one has
\begin{equation}\label{max 1}
\alpha\frac{\nabla_{i}\eta}{\eta}+\Psi^{'}\nabla_{i}\vartheta+\frac{\nabla_{i}h_{11}}{h_{11}}=0,
\end{equation}
\begin{equation*}\label{max 2}
\begin{aligned}
\alpha\left(\frac{\nabla_{i}\nabla_{j}\eta}{\eta}-\frac{\nabla_{i}\eta\nabla_{j}\eta}{\eta^{2}}\right)&+\Psi^{''}\nabla_{i}\vartheta\nabla_{j}\vartheta\\
&
+\Psi^{'}\nabla_{i}\nabla_{j}\vartheta+\frac{\nabla_{i}\nabla_{j}h_{11}}{h_{11}}-\frac{\nabla_{i}h_{11}\nabla_{j}h_{11}}{h_{11}^{2}}\leq
0.
\end{aligned}
\end{equation*}
Therefore, by Lemma \ref{F_{ij}'s equation 2}, we have
\begin{equation*}\label{max 3}
\begin{aligned}
0&\geq\alpha
F^{ij}\left(\frac{\nabla_{i}\nabla_{j}\eta}{\eta}-\frac{\nabla_{i}\eta\nabla_{j}\eta}{\eta^{2}}\right)+\Psi^{''}F^{ij}\nabla_{i}\vartheta\nabla_{j}\vartheta
+\Psi^{'}F^{ij}\nabla_{i}\nabla_{j}\vartheta\\
&\quad
+F^{ij}\frac{\nabla_{i}\nabla_{j}h_{11}}{h_{11}}-F^{ij}\frac{\nabla_{i}h_{11}\nabla_{j}h_{11}}{h_{11}^{2}}\\
&= \alpha
F^{ij}\left(\frac{\nabla_{i}\nabla_{j}\eta}{\eta}-\frac{\nabla_{i}\eta\nabla_{j}\eta}{\eta^{2}}\right)+\Psi^{''}F^{ij}\nabla_{i}\vartheta\nabla_{j}\vartheta
+\Psi^{'}F^{ij}\nabla_{i}\nabla_{j}\vartheta\\
&\quad
-fh_{11}+F^{ij}h_{im}h_{jm}+\frac{\nabla_{1}\nabla_{1}f}{h_{11}}-\frac{1}{h_{11}}F^{ij,
pq}\nabla_{1}h_{ij}\nabla_{1}h_{pq}-F^{ij}\frac{\nabla_{i}h_{11}\nabla_{j}h_{11}}{h_{11}^{2}}.
\end{aligned}
\end{equation*}
We also find that
\begin{equation*}\label{support function}
F^{ij}\nabla_{i}\nabla_{j}\vartheta=\vartheta
F^{ij}h_{im}h_{jm}+f+\nabla_{l}f\langle X, X_{l}\rangle_{L}.
\end{equation*}
Consequently,
\begin{equation}\label{max 4}
\begin{aligned}
0&\geq\alpha
F^{ij}\left(\frac{\nabla_{i}\nabla_{j}\eta}{\eta}-\frac{\nabla_{i}\eta\nabla_{j}\eta}{\eta^{2}}\right)+\Psi^{''}F^{ij}\nabla_{i}\vartheta\nabla_{j}\vartheta
+\Psi^{'}\nabla_{l}f\langle X, X_{l}\rangle_{L}-fh_{11}\\
&\quad
+\left(\Psi^{'}\vartheta+1\right)F^{ij}h_{im}h_{jm}+\frac{\nabla_{1}\nabla_{1}f}{h_{11}}-
\frac{1}{h_{11}}F^{ij,
pq}\nabla_{1}h_{ij}\nabla_{1}h_{pq}-F^{ij}\frac{\nabla_{i}h_{11}\nabla_{j}h_{11}}{h_{11}^{2}}.
\end{aligned}
\end{equation}
Since $f$ is convex in $\vartheta$, we have
\begin{equation*}
\nabla_{1}f=\frac{\partial f}{\partial
X^{\alpha}}\nabla_{1}X^{\alpha}+\frac{\partial
f}{\partial\vartheta}\nabla_{1}\vartheta,
\end{equation*}
\begin{equation*}
\begin{aligned}
\nabla_{1}\nabla_{1}f&=\frac{\partial^{2}f}{\partial
X^{\alpha}\partial X^{\beta}}\nabla_{1}X^{\alpha}\nabla_{1}X^{\beta}
+2\frac{\partial^{2}f}{\partial X^{\alpha}\partial\vartheta}\nabla_{1}X^{\alpha}\nabla_{1}\vartheta+\frac{\partial^{2}f}{\partial\vartheta^{2}}|\nabla_{1}\vartheta|^{2}\\
&\quad
+\frac{\partial f}{\partial X^{\alpha}}\nabla_{1}\nabla_{1}X^{\alpha}+\frac{\partial f}{\partial\vartheta}\nabla_{1}\nabla_{1}\vartheta\\
&\geq
\frac{\partial f}{\partial\vartheta}\nabla_{1}\nabla_{1}\vartheta-c_{3}h_{11}-c_{4}\\
&= \frac{\partial f}{\partial\vartheta}\left(\vartheta
h_{11}^{2}+\nabla_{l}h_{11}\langle X,
X_{l}\rangle_{L}\right)-c_{3}h_{11}-c_{4},
\end{aligned}
\end{equation*}
where $c_{3}$, $c_{4}$ are positive constants depending on
$||\varphi||_{C^{1}(\overline{M^{n}})}$ and
$||\psi||_{C^{2}\left(\overline{M^{n}}\times\left[\inf\limits_{\partial
M^{n}}u,\sup\limits_{\partial
M^{n}}u\right]\times\mathbb{R}\right)}$. Inserting this into
(\ref{max 4}) yields
\begin{equation}\label{max 5}
\begin{aligned}
0&\geq\alpha
F^{ij}\left(\frac{\nabla_{i}\nabla_{j}\eta}{\eta}-\frac{\nabla_{i}\eta\nabla_{j}\eta}{\eta^{2}}\right)+\Psi^{''}F^{ij}\nabla_{i}\vartheta\nabla_{j}\vartheta
+\Psi^{'}\nabla_{l}f\langle X, X_{l}\rangle_{L}+\left(\frac{\partial f}{\partial\vartheta}\cdot\vartheta-f\right)h_{11}\\
&\quad
+\left(\Psi^{'}\vartheta+1\right)F^{ij}h_{im}h_{jm}+\frac{\partial
f}{\partial\vartheta}\frac{\nabla_{l}h_{11}\langle X,
X_{l}\rangle_{L}}{h_{11}}
-\frac{1}{h_{11}}F^{ij, pq}\nabla_{1}h_{ij}\nabla_{1}h_{pq}\\
&\quad
-F^{ij}\frac{\nabla_{i}h_{11}\nabla_{j}h_{11}}{h_{11}^{2}}-c_{3},
\end{aligned}
\end{equation}
where we have assumed that $h_{11}$ is sufficiently large.
Otherwise, the assertion of Theorem \ref{main 1.2} holds.

Next, we assume that $\varphi$ has been extended to be constant in
the $\partial_{r}$ direction\footnote{~This can be assured, since
$\varphi$ is defined on $\overline{M^{n}}$ and of course one can
require its extension to the normal bundle of $M^{n}$ to be
constant.}. Therefore,
\begin{equation*}
\begin{aligned}
\nabla_{i}\nabla_{j}\eta&=
\sum_{\alpha,\beta=1}^{n}\frac{\partial^{2}\varphi}{\partial
X^{\alpha}\partial
X^{\beta}}\nabla_{i}X^{\alpha}\nabla_{j}X^{\beta}+
\sum_{\alpha=1}^{n}\frac{\partial\varphi}{\partial X^{\alpha}}\nabla_{i}\nabla_{j}X^{\alpha}-u_{ij}\\
&\geq \sum_{\alpha=1}^{n}\frac{\partial\varphi}{\partial
X^{\alpha}}\nu^{\alpha}h_{ij}-c_{5}h_{ij}v,
\end{aligned}
\end{equation*}
where $c_{5}>0$ depends on $||\varphi||_{C^{1}(\overline{M^{n}})}$
and we have again used Gaussian formula and the assumption that
$\varphi$ is affine. Consequently,
\begin{equation}\label{max 6}
F^{ij}\nabla_{i}\nabla_{j}\eta\geq\left(\sum_{\alpha=1}^{n}\frac{\partial\varphi}{\partial
X^{\alpha}}\nu^{\alpha}-c_{5}v\right)F^{ij}h_{ij}\geq-c_{6},
\end{equation}
where positive constant $c_{6}$ depends on $c_{5}$,
$||\psi||_{C^{0}\left(\overline{M^{n}}\times\left[\inf\limits_{\partial
M^{n}}u,\sup\limits_{\partial
M^{n}}u\right]\times\mathbb{R}\right)}$ and
$||\varphi||_{C^{1}(\overline{M^{n}})}$. Combining (\ref{max 5}) and
(\ref{max 6}), at $X_{0}$, we have
\begin{equation}\label{max 7}
\begin{aligned}
0&\geq-\frac{c_{6}\alpha}{\eta}-\alpha
F^{ij}\frac{\nabla_{i}\eta\nabla_{j}\eta}{\eta^{2}}+\Psi^{''}F^{ij}\nabla_{i}\vartheta\nabla_{j}\vartheta
+\Psi^{'}\nabla_{l}f\langle X, X_{l}\rangle_{L}+\left(\frac{\partial f}{\partial\vartheta}\cdot\vartheta-f\right)h_{11}\\
&\quad
+\left(\Psi^{'}\vartheta+1\right)F^{ij}h_{im}h_{jm}+\frac{\partial
f}{\partial\vartheta}\frac{\nabla_{l}h_{11}\langle X,
X_{l}\rangle_{L}}{h_{11}}
-\frac{1}{h_{11}}F^{ij, pq}\nabla_{1}h_{ij}\nabla_{1}h_{pq}\\
&\quad
-F^{ij}\frac{\nabla_{i}h_{11}\nabla_{j}h_{11}}{h_{11}^{2}}-c_{3}.
\end{aligned}
\end{equation}
We now estimate the remaining terms in (\ref{max 7}), and divide the
argument into two cases.

\textbf{Case 1}. Assume that there exists a positive constant
$\zeta$ to be determined such that
\begin{equation}\label{case 1.1}
h_{nn}\leq -\zeta h_{11}.
\end{equation}
Using the critical point condition (\ref{max 1}), we have
\begin{equation*}\label{case 1.2}
\begin{aligned}
F^{ij}\frac{\nabla_{i}h_{11}\nabla_{j}h_{11}}{h_{11}^{2}}&=
F^{ij}\left(\alpha\frac{\nabla_{i}\eta}{\eta}+\Psi^{'}\nabla_{i}\vartheta\right)\left(\alpha\frac{\nabla_{j}\eta}{\eta}+\Psi^{'}\nabla_{j}\vartheta\right)\\
&\leq
(1+\varepsilon^{-1})\alpha^{2}F^{ij}\frac{\nabla_{i}\eta\nabla_{j}\eta}{\eta^{2}}+(1+\varepsilon)(\Psi^{'})^{2}F^{ij}\nabla_{i}\vartheta\nabla_{j}\vartheta
\end{aligned}
\end{equation*}
for any $\varepsilon>0$. Since $|\nabla\eta|\leq
c_{7}(\widetilde{M^{n}})$, so
\begin{equation*}\label{case 1.3}
F^{ij}\frac{\nabla_{i}\eta\nabla_{j}\eta}{\eta^{2}}\leq
c_{8}\frac{\tr F^{ij}}{\eta^{2}},
\end{equation*}
where $c_{8}>0$ depends on $c_{7}$. Therefore, at $X_{0}$, we have
\begin{equation}\label{case 1.4}
\begin{aligned}
0&\geq-\frac{c_{6}\alpha}{\eta}-c_{9}\left[\alpha+(1+\varepsilon^{-1})\alpha^{2}\right]\frac{\tr F^{ij}}{\eta^{2}}+\left[\Psi^{''}-(1+\varepsilon)(\Psi^{'})^{2}\right]F^{ij}\nabla_{i}\vartheta\nabla_{j}\vartheta\\
&\quad +\left(\frac{\partial
f}{\partial\vartheta}\cdot\vartheta-f\right)h_{11}
+\left(\Psi^{'}\vartheta+1\right)F^{ij}h_{im}h_{jm}-c_{3}\\
&\quad +\frac{\partial
f}{\partial\vartheta}\frac{\nabla_{l}h_{11}\langle X,
X_{l}\rangle_{L}}{h_{11}}+\Psi^{'}\nabla_{l}f\langle X,
X_{l}\rangle_{L},
\end{aligned}
\end{equation}
where $c_{9}:=\max\{1, c_{8}\}$ and the concavity of $F(A)$ has been
used. On the other hand, from (\ref{max 1}), the last two terms of
the RHS of (\ref{case 1.4}) are bounded from below
\begin{equation*}\label{case 1.5}
\begin{aligned}
&\quad
\frac{\partial f}{\partial\vartheta}\frac{\nabla_{l}h_{11}\langle X, X_{l}\rangle_{L}}{h_{11}}+\Psi^{'}\nabla_{l}f\langle X, X_{l}\rangle_{L}\\
&= \left(\Psi^{'}\nabla_{l}f-\alpha\frac{\partial
f}{\partial\vartheta}\frac{\nabla_{l}\eta}{\eta}
-\frac{\partial f}{\partial\vartheta}\Psi^{'}\nabla_{l}\vartheta\right)\langle X, X_{l}\rangle_{L}\\
&=
\left(\Psi^{'}\frac{\partial f}{\partial X^{\beta}}\nabla_{l}X^{\beta}-\alpha\frac{\partial f}{\partial\vartheta}\frac{\nabla_{l}\eta}{\eta}-\frac{\partial f}{\partial\vartheta}\Psi^{'}\nabla_{l}\vartheta\right)\langle X, X_{l}\rangle_{L}\\
&\geq -\frac{c_{10}\alpha}{\eta}-c_{11},
\end{aligned}
\end{equation*}
where $c_{10}$ is a positive constant depending on $c_{7}$,
$||\varphi||_{C^{1}(\overline{M^{n}})}$,
$||\psi||_{C^{1}\left(\overline{M^{n}}\times\left[\inf\limits_{\partial
M^{n}}u,\sup\limits_{\partial
M^{n}}u\right]\times\mathbb{R}\right)}$, and $c_{11}>0$ depends on
$||\varphi||_{C^{1}(\overline{M^{n}})}$,
$||\psi||_{C^{1}\left(\overline{M^{n}}\times\left[\inf\limits_{\partial
M^{n}}u,\sup\limits_{\partial
M^{n}}u\right]\times\mathbb{R}\right)}$. Therefore
\begin{equation}\label{case 1.6}
\begin{aligned}
0&\geq-\frac{c_{12}\alpha}{\eta}-c_{9}\left[\alpha+(1+\varepsilon^{-1})\alpha^{2}\right]\frac{\tr F^{ij}}{\eta^{2}}+\left[\Psi^{''}-(1+\varepsilon)(\Psi^{'})^{2}\right]F^{ij}\nabla_{i}\vartheta\nabla_{j}\vartheta\\
&\quad +\left(\frac{\partial
f}{\partial\vartheta}\cdot\vartheta-f\right)h_{11}
+\left(\Psi^{'}\vartheta+1\right)F^{ij}h_{im}h_{jm}-c_{13},
\end{aligned}
\end{equation}
where constant $c_{12}>0$ depends on $c_{6}$, $c_{10}$, and constant
$c_{13}>0$ depends on $c_{3}$ and $c_{11}$. By the Weingarten
formula (\ref{Wein for}), it follows that
$$F^{ij}\nabla_{i}\vartheta\nabla_{j}\vartheta=F^{ij}h_{il}h_{jk}\langle X, X_{l}\rangle_{L}\langle X, X_{k}\rangle_{L}\leq c_{14}F^{ij}h_{il}h_{jk},$$
where $c_{14}$ is a positive constant depending on
$||\varphi||_{C^{1}(\overline{M^{n}})}$, and then we can take a
function $\Psi$ satisfying
\begin{equation}\label{case 1.7}
\Psi^{''}-(1+\varepsilon)(\Psi^{'})^{2}\leq 0.
\end{equation}
Since $M^{n}$ is bounded and $C^2$, there exists a positive constant
$a=a(\rho)>\sup\limits_{M^{n}}u$ such that
$$-a\leq\vartheta<-\sup_{M^{n}}u.$$
Let us take
$$\Psi(\vartheta)=-\log(2a+\vartheta),$$
so we have (\ref{case 1.7}) and
$$\Psi^{'}\vartheta+1+c_{14}(\Psi^{''}-(1+\varepsilon)(\Psi^{'})^{2})\geq\frac{1}{2}   \qquad \mathrm{for}~~\varepsilon\leq\frac{2a^{2}}{c_{14}}.$$
From (\ref{case 1.6}), together with
$$F^{ij}h_{im}h_{jm}=F^{ii}h_{ii}^{2}\geq\frac{\zeta^{2}}{n}h_{11}^{2}\tr F^{ij},$$
which follows from the assumption (\ref{case 1.1}) and the fact
$F^{nn}\geq\frac{1}{n}\tr F^{ij}$, at $X_{0}$, we have that
\begin{equation*}\label{case 1.8}
\begin{aligned}
0&\geq-\frac{c_{12}\alpha}{\eta}-c_{9}\left[\alpha+(1+\varepsilon^{-1})\alpha^{2}\right]\frac{\tr F^{ij}}{\eta^{2}}\\
&\qquad +\left(\frac{\partial
f}{\partial\vartheta}\cdot\vartheta-f\right)h_{11}+\frac{\zeta^{2}}{2n}h_{11}^{2}\tr
F^{ij}-c_{13},
\end{aligned}
\end{equation*}
which implies an upper bound
$$\eta h_{11}\leq\frac{c_{15}}{\zeta}  \qquad \mathrm{at} ~~X_{0},$$
since
$$\tr F^{ij}=\frac{(n-k+1)\sigma_{k-1}}{kf^{k-1}}>0,$$
where $c_{15}$ is a positive constant depending on $c_{9}$,
$c_{12}$, $c_{13}$, $\alpha$, $M^{n}$,
$||\varphi||_{C^{0}(\overline{M^{n}})}$.

\textbf{Case 2}. We now assume that
\begin{equation}\label{case 2.1}
h_{nn}\geq -\zeta h_{11}.
\end{equation}
Since $h_{11}\geq h_{22}\geq\cdots\geq h_{nn}$, we have
\begin{equation*}\label{case 2.2}
h_{ii}\geq -\zeta h_{11} \qquad \mathrm{for~~all}~~i=1,\cdots,n.
\end{equation*}

For a positive constant $\tau$, assume to be $4$, we divide
$\{1,\cdots,n\}$ into two parts as follows
$$I=\{i:\mathcal{P}^{ii}\leq 4\mathcal{P}^{11}\}, \qquad J=\{j:\mathcal{P}^{jj}>4\mathcal{P}^{11}\},$$
where $\mathcal{P}^{ii}:=\frac{\partial \mathcal{P}}{\partial
h_{ii}}=\mathcal{P}_{i}$ is evaluated at $\lambda(X_{0})$. Then for
each $i\in I$, by (\ref{max 1}), we have
\begin{equation*}\label{case 2.3}
\begin{aligned}
\mathcal{P}_{i}\frac{|\nabla_{i}h_{11}|^{2}}{h_{11}^{2}}&=
\mathcal{P}_{i}\left(\alpha\frac{\nabla_{i}\eta}{\eta}+\Psi^{'}\nabla_{i}\vartheta\right)^{2}\\
&\leq
(1+\varepsilon^{-1})\alpha^{2}\mathcal{P}_{i}\frac{|\nabla_{i}\eta|^{2}}{\eta^{2}}+(1+\varepsilon)(\Psi^{'})^{2}\mathcal{P}_{i}|\nabla_{i}\vartheta|^{2}
\end{aligned}
\end{equation*}
for any $\varepsilon>0$.  For each $j\in J$, we have
\begin{equation*}\label{case 2.4}
\begin{aligned}
\alpha \mathcal{P}_{j}\frac{|\nabla_{j}\eta|^{2}}{\eta^{2}}&=
\alpha^{-1}\mathcal{P}_{j}\left(\frac{\nabla_{j}h_{11}}{h_{11}}+\Psi^{'}\nabla_{j}\vartheta\right)^{2}\\
&\leq
\frac{1+\varepsilon}{\alpha}(\Psi^{'})^{2}\mathcal{P}_{j}|\nabla_{j}\vartheta|^{2}+\frac{1+\varepsilon^{-1}}{\alpha}\mathcal{P}_{j}\frac{|\nabla_{j}h_{11}|^{2}}{h_{11}^{2}}
\end{aligned}
\end{equation*}
for any $\varepsilon>0$. Consequently,
\begin{equation*}\label{case 2.5}
\begin{aligned}
&\qquad
\alpha\sum_{i=1}^{n}\mathcal{P}_{i}\frac{|\nabla_{i}\eta|^{2}}{\eta^{2}}+\sum_{i=1}^{n}\mathcal{P}_{i}\frac{|\nabla_{i}h_{11}|^{2}}{h_{11}^{2}}\\
&\leq
\left[\alpha+(1+\varepsilon^{-1})\alpha^{2}\right]\sum_{i\in I}\mathcal{P}_{i}\frac{|\nabla_{i}\eta|^{2}}{\eta^{2}}+(1+\varepsilon)(\Psi^{'})^{2}\sum_{i\in I}\mathcal{P}_{i}|\nabla_{i}\vartheta|^{2}\\
&\qquad
+\frac{1+\varepsilon}{\alpha}(\Psi^{'})^{2}\sum_{j\in J}\mathcal{P}_{j}|\nabla_{j}\vartheta|^{2}+\left[1+(1+\varepsilon^{-1})\alpha^{-1}\right]\sum_{j\in J}\mathcal{P}_{j}\frac{|\nabla_{j}h_{11}|^{2}}{h_{11}^{2}}\\
&\leq
4n\left[\alpha+(1+\varepsilon^{-1})\alpha^{2}\right]\mathcal{P}_{1}\frac{|\nabla_{i}\eta|^{2}}{\eta^{2}}+(1+\varepsilon)(1+\alpha^{-1})(\Psi^{'})^{2}\sum_{i=1}^{n}\mathcal{P}_{i}|\nabla_{i}\vartheta|^{2}\\
&\qquad +\left[1+(1+\varepsilon^{-1})\alpha^{-1}\right]\sum_{j\in
J}\mathcal{P}_{j}\frac{|\nabla_{j}h_{11}|^{2}}{h_{11}^{2}}.
\end{aligned}
\end{equation*}
Using this estimate and (\ref{max 7}), the following inequality
\begin{equation*}\label{case 2.6}
\begin{aligned}
0&\geq
-\frac{c_{6}\alpha}{\eta}-4n\left[\alpha+(1+\varepsilon^{-1})\alpha^{2}\right]\mathcal{P}_{1}\frac{|\nabla_{i}\eta|^{2}}{\eta^{2}}
+\left[\Psi^{''}-(1+\varepsilon)(1+\alpha^{-1})(\Psi^{'})^{2}\right]\mathcal{P}_{i}|\nabla_{i}\vartheta|^{2}\\
&\qquad +\Psi^{'}\nabla_{l}f\langle X,
X_{l}\rangle_{L}+\left(\frac{\partial
f}{\partial\vartheta}\cdot\vartheta-f\right)h_{11}
+(\Psi^{'}\vartheta+1)F^{ij}h_{im}h_{jm}+\frac{\partial f}{\partial\vartheta}\frac{\nabla_{l}h_{11}\langle X, X_{l}\rangle_{L}}{h_{11}}\\
&\qquad
-\frac{1}{h_{11}}F^{ij,pq}\nabla_{1}h_{ij}\nabla_{1}h_{pq}-\left[1+(1+\varepsilon^{-1})\alpha^{-1}\right]\sum_{j\in
J}\mathcal{P}_{j}\frac{|\nabla_{j}h_{11}|^{2}}{h_{11}^{2}}-c_{13}
\end{aligned}
\end{equation*}
holds at $X_{0}$. Then as \textbf{Case 1}, we have that for an
appropriate selection of $\Psi$,
\begin{equation}\label{case 2.7}
\begin{aligned}
0&\geq
-\frac{c_{12}\alpha}{\eta}-c_{16}(\alpha+\alpha^{2})\frac{\mathcal{P}_{1}}{\eta^{2}}+\frac{1}{2n}\mathcal{P}_{1}h_{11}^{2}+\left(\frac{\partial f}{\partial\vartheta}\cdot\vartheta-f\right)h_{11}-c_{13}\\
&\qquad
-\frac{1}{h_{11}}F^{ij,pq}\nabla_{1}h_{ij}\nabla_{1}h_{pq}-\left[1+c_{17}\alpha^{-1}\right]\sum_{j\in
J}\mathcal{P}_{j}\frac{|\nabla_{j}h_{11}|^{2}}{h_{11}^{2}},
\end{aligned}
\end{equation}
where $c_{16}>0$ depends on $n$, $\varepsilon^{-1}$, and
$c_{17}=(1+\varepsilon^{-1})$.

We \textbf{claim} that
\begin{equation}\label{case 2.8}
-\frac{1}{h_{11}}F^{ij,pq}\nabla_{1}h_{ij}\nabla_{1}h_{pq}-\left[1+c_{17}\alpha^{-1}\right]\sum_{j\in
J}\mathcal{P}_{j}\frac{|\nabla_{j}h_{11}|^{2}}{h_{11}^{2}}\geq 0.
\end{equation}
If the \textbf{claim} (\ref{case 2.8}) holds, then from (\ref{case
2.7}) we have
$$\left(\frac{\partial f}{\partial\vartheta}\cdot\vartheta-f\right)h_{11}+\frac{1}{2n}\mathcal{P}_{1}h_{11}^{2}\leq c_{18}(1+\frac{1}{\eta}+\frac{\mathcal{P}_{1}}{\eta^{2}}),$$
from which we again get a bound for $\eta h_{11}$ at $X_{0}$ due to
condition (\ref{f's strictly condition}), where $c_{18}>0$ depends
on $c_{12}$, $c_{13}$, $c_{16}$, $c_{17}$ and $\alpha$.

We now prove the \textbf{claim}. Using the concavity of
$\mathcal{P}$, Lemma \ref{sym} and the Codazzi equation
(\ref{Codazzi-1}), we can obtain
$$-\frac{1}{h_{11}}F^{ij,pq}\nabla_{1}h_{ij}\nabla_{1}h_{pq}\geq -\frac{2}{h_{11}}\sum_{j\in J}\frac{\mathcal{P}_{1}-\mathcal{P}_{j}}{\lambda_{1}-\lambda_{j}}|\nabla_{j}h_{11}|^{2}.$$
We then need to show that
$$-\frac{2(\mathcal{P}_{1}-\mathcal{P}_{j})}{h_{11}(\lambda_{1}-\lambda_{j})}\geq (1+c_{17}\alpha^{-1})\frac{\mathcal{P}_{j}}{h_{11}^{2}} \qquad \mathrm{for~~each}~~j\in J$$
provided that $\alpha$ is sufficiently large.

Set $\delta=c_{17}\alpha^{-1}$, and then we need to show
\begin{equation}\label{case 2.9}
(1-\delta)\mathcal{P}_{j}\lambda_{1}\geq
2\mathcal{P}_{1}\lambda_{1}-(1+\delta)\mathcal{P}_{j}\lambda_{j}
\qquad\qquad \mathrm{for}~~j\in J
\end{equation}
provided $\delta>0$ is sufficiently small. We show this if either
$\lambda_{j}\geq 0$ or $\lambda_{j}\leq 0$ and
$|\lambda_{j}|\leq\zeta\lambda_{1}$ for a sufficiently small
positive constant $\zeta$.

Since $j\in J$, so we have $\mathcal{P}_{j}>4\mathcal{P}_{1}$.
Therefore, if $\lambda_{j}\geq 0$, then (\ref{case 2.9}) is
satisfied if $\delta=1/4$. On the other hand, if $\lambda_{j}\leq
0$, then $|\lambda_{j}|\leq\zeta\lambda_{1}$ by (\ref{case 2.1}),
and therefore (\ref{case 2.9}) is again satisfied if $\delta=1/4$
and $\zeta=1/5$.

The proof of Theorem \ref{main 1.2} is finished.
\end{proof}

\section{Existence and uniqueness} \label{se5}

At end, we can show the existence and uniqueness of solutions to the
PCP (\ref{main equations}) as follows:

\begin{proof} [Proof of Theorem \ref{main 1.3}]
Clearly, the PCP (\ref{main equations}) is equivalent with the
following Dirichlet problem
\begin{equation*}\label{main equations 2}
\left\{
\begin{aligned}
&\sigma_{k}(u,D u,D^{2}u)=\psi(x,u,\vartheta(u,D u)), \qquad &&x\in
M^{n}\subset\mathbb{R}^{n+1}_{1},
\\
&u=\varphi, \qquad&&x\in \partial M^{n},
\end{aligned}
\right.
\end{equation*}
and the method of continuity can be used to get the existence of its
solutions. We divide the argument into three steps as follows:

\textbf{Step 1}. For each $t\in[0,1]$, consider the following
problem\footnote{~Clearly, the operator $\Delta$ in the Dirichlet
problem (\ref{main equations 3}) should be the Laplacian on
$M^{n}\subset\mathscr{H}^{n}(1)$. In fact, this happens to all
symbols $\Delta$ in Section \ref{se5}. For convenience and if
without confusion, we abuse the notation $\Delta$, which in this
paper was used to stand for the Laplacian on different geometric
objects (i.e., on the convex piece $M^{n}$ or the spacelike graphic
hypersurface $\mathcal{G}$).}
\begin{equation}\label{main equations 3}
\left\{
\begin{aligned}
&t\sigma_{k}(u,D u,D^{2}u)+(1-t)\Delta u=\psi(x,u,\vartheta(u,D u)),
\qquad &&x\in M^{n},
\\
&u=\varphi, \qquad&&x\in \partial M^{n}.
\end{aligned}
\right.
\end{equation}
Clearly, for $t=0$, (\ref{main equations 3}) corresponds to the
Dirichlet problem of the Laplace operator. Let $\omega=u-\varphi$,
and then (\ref{main equations 3}) is equivalent to
\begin{equation}\label{main equations 4}
\left\{
\begin{aligned}
&t\sigma_{k}(\omega+\varphi,D(\omega+\varphi),D^{2}(\omega+\varphi))+(1-t)
\Delta(\omega+\varphi) \\
& \qquad\qquad\qquad\qquad\qquad\qquad
=\psi(x,\omega+\varphi,\vartheta((\omega+\varphi),D(\omega+\varphi))),
\quad &&x\in M^{n},
\\
&\omega=0, \quad&&x\in \partial M^{n}.
\end{aligned}
\right.
\end{equation}

Now, we set
$$\mathcal{X}:=\left\{\omega\in C^{2,\alpha}(\overline{M^{n}})|\omega=0~~\mathrm{on}~~\partial M^{n}\right\}$$
and
$$\mathcal{F}(\omega,t):=t\sigma_{k}(\omega+\varphi,D(\omega+\varphi),D^{2}(\omega+\varphi))+(1-t)
\Delta(\omega+\varphi)-\psi\left(x,\omega+\varphi,\vartheta((\omega+\varphi),D(\omega+\varphi))\right).$$
Then the solvability of (\ref{main equations 4}) is equivalent to
find a function $\omega\in\mathcal{X}$ such that
$\mathcal{F}(\omega,t)=0$ in $M^{n}$.

Set
$$I=\{t\in[0,1]| \mathrm{~there~~exists~~a}~~\omega\in\mathcal{X}~~\mathrm{such~~that}~~\mathcal{F}(\omega,t)=0\}.$$
By the standard Schauder theory for the Laplace operator (see, e.g.,
\cite[Chap. 5]{HL}), we know that $0\in I$. The rest is to show
$1\in I$. To do this, we need to prove that $I$ is both open and
closed in $[0,1]$.

\textbf{Step 2}. We first show that $I$ is open. Note that
$\mathcal{F}:\mathcal{X}\times[0,1]\rightarrow
C^{\alpha}(\overline{M^{n}})$ is of class $C^{1}$ and using its
Fr\`{e}chet derivative, we have a uniformly elliptic operator with
$C^{\alpha}$-coefficients. The Fr\`{e}chet derivative here is given
by
$$\mathcal{F}_{\omega}(\omega,t)(\theta):=\lim_{\varepsilon\rightarrow 0}\frac{\mathcal{F}(\omega+\varepsilon\theta,t)-\mathcal{F}(\omega,t)}{\varepsilon}.$$
By the linear Schauder theory, $\mathcal{F}_{\omega}(\omega,t)$ is
an invertible operator from $\mathcal{X}$ to
$C^{\alpha}(\overline{M^{n}})$. Suppose $t_{0}\in I$, i.e.,
$\mathcal{F}(\omega^{t_{0}},t_{0})=0$ for some
$\omega^{t_{0}}\in\mathcal{X}$. By the implicit function theorem,
for any $t$ close to $t_{0}$, there is a unique
$\omega^{t}\in\mathcal{X}$, close to $\omega^{t_{0}}$ in the
$C^{2,\alpha}$-norm, satisfying $\mathcal{F}(\omega^{t},t)=0$. Hence
$t\in I$ for all such $t$, and so $I$ is open.

\textbf{Step 3}. For the closedness, by the lower order estimates in
Section \ref{c1es}, the curvature estimates in Section \ref{S3}
(i.e., Theorems \ref{main1.1}, \ref{main 1.2}) and boundary $C^{2}$
estimates (which correspond to the special case $l=0$ of the $C^{2}$
boundary estimates given in \cite[Section 6]{ggm}), we know that any
$\omega$ in $\mathcal{X}$ of $\mathcal{F}(\omega,t)=0$ in
$\overline{M^{n}}$ satisfies a uniform $C^{2,\alpha}$-estimate,
independent of $t$, i.e.,
$$|\omega^{t}|_{C^{2,\alpha}(\overline{M^{n}})}\leq C, \qquad \mathrm{independent~~of}~~t.$$
Using Arzel\`{a}-Asoli theorem, the closedness of $I$ follows
directly.

Therefore, by the above argument, we know that $I$ is the whole unit
interval. Then the function $\omega^{1}$ is our desired solution of
(\ref{main equations 4}) corresponding to $t=1$. The uniqueness of
solutions to the PCP (\ref{main equations}) can be obtained by
directly using the comparison principle to the $\sigma_{k}$
operator. This completes the proof.
\end{proof}

\section*{Acknowledgments}
This work is partially supported by the NSF of China (Grant Nos.
11801496 and 11926352), the Fok Ying-Tung Education Foundation
(China) and  Hubei Key Laboratory of Applied Mathematics (Hubei
University).

\end{document}